\documentclass{article}
\usepackage{amsmath}
\usepackage{amsfonts}
\usepackage{amsthm}
\usepackage{amssymb}
\usepackage{mathrsfs}
\usepackage{hyperref}

\theoremstyle{definition}

\newtheorem{example}{Example}

\theoremstyle{remark}
\newtheorem*{remark}{Remark}

\theoremstyle{plain}
\newtheorem{theorem}{Theorem}
\newtheorem{corollary}{Corollary}
\newtheorem{lemma}{Lemma}

\DeclareMathOperator{\Alg}{Alg}
\DeclareMathOperator{\im}{Im}
\DeclareMathOperator{\re}{Re}
\DeclareMathOperator{\codim}{codim}
\DeclareMathOperator{\defeq}{\stackrel{\text{def\textsuperscript{\underline{n}}}}{=}}
\DeclareMathOperator{\thm}{\text{Th\textsuperscript{\underline{m}} }}

\begin{document}

\title{Multipliers of Hilbert Spaces of Analytic Functions on the Complex Half-Plane}
\author{Andrzej S. Kucik}
\date{}
\maketitle

\begin{flushright}
School of Mathematics \\
University of Leeds \\
Leeds LS2 9JT \\
United Kingdom \\
e-mail: \ttfamily{mmask@leeds.ac.uk}

\end{flushright}

\begin{abstract}
It follows, from a generalised version of Paley-Wiener theorem, that the Laplace transform is an isometry between certain spaces of weighted $L^2$ functions defined on $(0, \infty)$ and (Hilbert) spaces of analytic functions on the right complex half-plane (for example Hardy, Bergman or Dirichlet spaces). We can use this fact to investigate properties of \emph{multipliers} and \emph{multiplication operators} on the latter type of spaces. In this paper we present a full characterisation of multipliers in terms of a generalised concept of a \emph{Carleson measure}. Under certain conditions, these spaces of analytic functions are not only Hilbert spaces but also Banach algebras, and are therefore contained within their spaces of multipliers. We provide some necessary as well as sufficient conditions for this to happen and look at its consequences. 
\end{abstract}

\section{Introduction and notation}
Banach spaces of analytic functions defined on the unit disk of the complex plane and operators acting on them have been studied in great detail for the past hundred years (most famous of them being the Hardy spaces $H^p$, \cite{duren1970}, \cite{hoffman1962}, \cite{mashreghi2009}), and their properties are well understood. Many of them can easily be applied to more general regions of the complex plane, however it is not always possible if we also consider regions of infinite measure, for example a complex half-plane with the Lebesgue area measure.

Let $\tilde{\nu}$ be a positive regular Borel measure on $[0, \infty)$ satisfying the following ($\Delta_2$)-condition:
\begin{displaymath}
\sup_{r>0} \frac{\tilde{\nu}[0, 2r)}{\tilde{\nu}[0, r)} < \infty.
\end{displaymath}
Let also $\nu$ be a positive regular Borel measure on $\overline{\mathbb{C}_+} = [0, \infty) \times i\mathbb{R}$ (the closed right complex-half plane) given by $d\nu = d\tilde{\nu} \otimes d\lambda$, where $\lambda$ denotes the Lebesgue measure. For $1 \leq p < \infty$ a \emph{Zen space on} $\mathbb{C}_+$ is the space
\begin{displaymath}
	A^p_\nu = \left\{F : \mathbb{C}_+ \longrightarrow \mathbb{C} \; \text{analytic} \; : \; \left\|F\right\|^p_{A^p_\nu} := \sup_{\varepsilon>0} \int_{\overline{\mathbb{C}_+}} \left|F(z+\varepsilon)\right|^p \, d\nu(z) < \infty \right\}.
\end{displaymath}
It was introduced in \cite{jacob2013}, and named after Zen Harper who constructed it in \cite{zen2009} and \cite{zen2010}. Evidently $A^2_\nu$ is a Hilbert space. There are several well-known examples of Zen spaces, such as the Hardy spaces $H^p(\mathbb{C}_+)$ (where $\tilde{\nu}$ is the Dirac measure in 0) \cite{duren1970}, \cite{hoffman1962}, \cite{mashreghi2009}, or the weighted Bergman spaces $\mathcal{B}^p_\alpha(\mathbb{C}_+)$ (where $d\tilde{\nu}(t) = t^\alpha dt, \; \alpha > -1$) \cite{duren2004}, \cite{hedenmalm2000}. In the Hilbertian setting, $A^2_\nu$ spaces of functions on the complex-half plane might be viewed as "continuous" counterparts (in some sense) of spaces of analytic functions on the complex unit disk. We notice that
\begin{displaymath}
	\left\|f\right\|^2_{H^2(\mathbb{D})} = \sum_{n=0}^\infty \|a_n\|^2 \; \; \; \; \; (f(z)=\sum_{n=0}^\infty a_n z^n \in H^2(\mathbb{D}))
	\end{displaymath}
or
\begin{displaymath}
	\left\|f\right\|^2_{\mathcal{B}^2_0(\mathbb{D})} = \sum_{n=0}^\infty \frac{\|a_n\|^2}{n+1} \; \; \; \; \; (f(z)=\sum_{n=0}^\infty a_n z^n \in \mathcal{B}^2_0(\mathbb{D})).
	\end{displaymath}
Whereas
\begin{displaymath}
	\left\|F\right\|^2_{H^2(\mathbb{C}_+)} = \int_0^\infty |f(t)|^2 \, dt  \; \; (F(z)=\int_0^\infty f(t) e^{-tz} \, dt \in H^2(\mathbb{C}_+), \, f \in L^2(0, \infty))
	\end{displaymath}
and
\begin{displaymath}
	\left\|F\right\|^2_{\mathcal{B}_0^2(\mathbb{C}_+)} = \int_0^\infty |f(t)|^2 \, \frac{dt}{t} \; \; \; \; \; (F(z)=\int_0^\infty f(t) e^{-tz} \, dt \in \mathcal{B}^2_0(\mathbb{C}_+), \, f \in L^2_{\frac{1}{t}}(0, \infty).
	\end{displaymath}
It can be shown that the Laplace transform defines an isometric map from certain weighted $L^2_w(0, \infty)$ spaces into $A^2_\nu$ \cite{jacob2013}, and thus indeed, the Zen spaces are half-plane equivalents of some of the weighted Hardy spaces $H^2(\beta)$, that is the Hilbert space of analytic functions on the disk with
\begin{displaymath}
\|f\|^2_{H^2(\beta)} = \left\|\sum_{n=0}^\infty a_n z^n \right\|^2_{H^2(\beta)} := \sum_{n=0}^\infty \|a_n\| \beta(n)^2 < \infty,
\end{displaymath}
(for some real, positive sequence $(\beta(n))_{n=0}^\infty$, for details see for example \cite{chalendar2014} or \cite{cowen1995}); linking the weighted $L^2$ spaces on $(0, \infty)$ with spaces of analytic functions on the complex half-plane in an analogous way as spaces of analytic functions on the complex unit disk are linked to the weighted $\ell^2$ spaces. 

The Zen spaces, however, do not cover many important examples of spaces of analytic functions on the complex half-plane, such as the Dirichlet space or Hardy-Sobolev space, which are included in the definition of $H^2(\beta)$ on the complex unit disk. This is why in Section \ref{sec:pre} of this paper we present a construction of more general Banach spaces, $A^p(\mathbb{C}_+, (\nu_n)_{n=0}^m)$, of which the Zen spaces are a special case (along with Dirichlet, Hardy-Sobolev and many other types of spaces). We show the existence of an isometry between closed subspaces $A^2_{(m)}$ of $A^2(\mathbb{C}_+, (\nu_n)_{n=0}^m)$ and weighted $L^2$ spaces on $(0, \infty)$, again using the Laplace transform, and find their reproducing kernels. In Section \ref{sec:mult} we define and describe their multipliers, using the notion of a Carleson measure and its generalisations. The Carleson measures for the half-plane have important applications in control theory \cite{jacob2004}, particularly significant when dealing with controls lying in weighted $L^2$ spaces, with Laplace transforms in the spaces we consider in this article (for example Hardy-Sobolev space) \cite{jacob2013}. It was in fact control theory problems in engineering and the role of Laplace-Carleson embeddings in controllability and admissibility that lead	towards the study of Zen spaces \cite{jacob2014}. However, as said before, Zen spaces are often insufficient when dealing with certain problems, and this is the main motivation for studying $A^2_{(m)}$ spaces.

In Section \ref{sec:ba} we investigate some Banach algebras contained within $A^2_{(m)}$ and show that $A^2_{(m)}$ are sometimes, under certain conditions, Banach algebras themselves. And finally in Section \ref{sec:specideals} we state some results about the ideals of the Banach algebra of $\mathscr{M}(A^2_{(m)})$. We also raise some important questions regarding the maximal ideal spaces of $\mathscr{M}(A^2_{(m)})$ and the Corona Problem, which still remain to be answered.

Construction of the isometry in Section \ref{sec:pre} was presented in \cite{jacob2013}, for the Laplace transform and a Zen space. The characterisation of multipliers of the Dirichlet space on the disk ($\mathcal{D}$) in terms of Carleson measures was initially given by David Stegenga in \cite{stegenga1980}, where Carleson measures for $\mathcal{D}$ are also described. And the idea of Carleson measures themselves was formed by Lennart Carleson in his solution of the corona problem \cite{carleson1962} for $H^\infty$). A good and very recently published reference for the Dirichlet space on the disk ($\mathcal{D}$) is \cite{el-fallah2014}. There is extensive literature devoted to the study of Banach algebras, which are usually seen as spaces of bounded linear operators on some Hilbert space, but Banach algebras which are also Hilbert spaces are not considered very often (publications known to the author include a short article by Yu. N. Kuznetsova \cite{kuznetsova2006}, and some brief mentions in \cite{acta1995}, \cite{dales2000} and \cite{nikolskii1970}). We believe that the results presented in Section \ref{sec:ba} are mainly new, and were not published before, with the exception of the equation \eqref{eq:convo} (which has been given in \cite{acta1995}, \cite{nikolskii1970} and \cite{kuznetsova2006}) and Theorem \ref{thm:bal} (which has been known, to some extent, for $m=0,1$). The last theorem of Section \ref{sec:specideals} is a variant of a well-known theorem from \cite{hoffman1962}, but its proof had to be altered substantially, due to difficulties arising from introduction of derivatives and unbounded domain.

\section{Preliminaries}
\label{sec:pre}
One of the most fundamental tools used to study Zen spaces with $p=2$ is the fact that the Laplace transform defines an isometric map $\mathfrak{L} : L^2_w (0, \infty) \longrightarrow A^2_\nu$, where
\begin{displaymath}
	w(t) = 2\pi \int_0^\infty e^{-2rt} \, d\tilde{\nu}(r) \; \; \; \; \; (t > 0)
\end{displaymath} 
(see \cite{jacob2013}). In fact, we can extend this result to study more general spaces on the complex half-plane.
\begin{theorem}
\label{thm:mainthm}
The $n$-th derivative of the Laplace transform defines an isometric map $\mathfrak{L}^{(n)} : L^2_{w_n} (0, \infty) \longrightarrow A^2_\nu$, where
\begin{equation}
	w_n(t) = 2\pi t^{2n}\int_0^\infty e^{-2rt} \, d\tilde{\nu}(r) \; \; \; \; \; (t > 0).
\label{eq:w}
\end{equation}
(Here $\nu_n$ is defined in the same way as $\nu$ above).
\end{theorem}

\begin{proof}
The proof follows closely the proof of Proposition 2.3 in \cite{jacob2013}, using the elementary relation between the Laplace and the Fourier transform ($\mathfrak{F}$), and that the latter defines an isometry (by the Plancherel theorem \cite{mashreghi2009}). Let $f \in L_{w_n}^2 (0, \infty)$, $g_n(t)=t^n f(t)$ and $z=r+is \in \mathbb{C}_+$. Then
\begin{align*}
\sup_{\varepsilon>0} \int_{\mathbb{C}_+} \left| \mathfrak{L}^{(n)} [f](z +\varepsilon)\right|^2 d\nu(z) &= \sup_{\varepsilon>0} \int_0^\infty \int_{-\infty}^\infty \left| \mathfrak{L}^{(n)} [f](r+is+\varepsilon)\right|^2 d\lambda(s) d\tilde{\nu}(r) \\
&= \sup_{\varepsilon>0} \int_0^\infty \left\| (-1)^n \mathfrak{L}[t^n f] (r+is+\varepsilon) \right\|_{L^2(i\mathbb{R})}^2 d\tilde{\nu} (r) \\
&= \sup_{\varepsilon>0} \int_0^\infty \left\| \mathfrak{L}[g_n] (r+is+\varepsilon) \right\|_{L^2(i\mathbb{R})}^2 \, d\tilde{\nu} (r) \\
&= \sup_{\varepsilon>0} \int_0^\infty \left\| \mathfrak{F}\left[e^{-(r+\varepsilon)t} g_n \right]\right\|_{L^2(\mathbb{R})}^2 \, d\tilde{\nu} (r) \\
&= \sup_{\varepsilon>0} \int_0^\infty 2\pi \left\| e^{-(r+\varepsilon)t} g_n \right\|_{L^2(0, \infty)}^2 \, d\tilde{\nu} (r) \\
&= \sup_{\varepsilon>0} \int_0^\infty \left| g_n(t) \right|^2 \, 2\pi \int_0^\infty e^{-2(r+\varepsilon)t} \, d\tilde{\nu} (r) \, dt \\
&\stackrel{\eqref{eq:w}}{=} \int_0^\infty \left| f(t) \right|^2 w_n(t) \, dt.
\end{align*}
\end{proof}

This result allows us to generalise the notion of a Zen space, defining a new Hilbert space. First, let $(\tilde{\nu}_n)_{n=0}^m$ be a sequence (not necessarily finite) of positive regular Borel measures on $[0, \infty)$ satisfying the $(\Delta_2)$-condition and let $\nu_n$ be a measure on $\overline{\mathbb{C}_+}$ given by $d\nu_n = d\tilde{\nu}_n \otimes d\lambda$ (where $\lambda$ is as above, the Lebesgue measure). Set
\begin{displaymath}
	A^p\left(\mathbb{C}_+, \, (\nu_n)_{n=0}^m\right) = \left\{F : \mathbb{C}_+ \longrightarrow \mathbb{C} \; : \; F^{(n)} \in A^p_{\nu_n}, \; \forall \, 0 \leq n \leq m \right\}.
\end{displaymath}
It is clearly a Banach space, with respect to the norm given by
\begin{displaymath}
	\left\|F\right\|_{A^p\left(\mathbb{C}_+, \, (\nu_n)_{n=0}^m\right)} : = \left( \sum_{n=0}^m \left\|F^{(n)}\right\|^p_{A^p_{\nu_n}} \right)^{1/p} \; \; \; \; \; \left(F \in A^p_{(m)}\right).
\end{displaymath}
Analogously, if $p=2$, it is a Hilbert space, with the inner product given by
	\begin{displaymath}
	\left\langle F, \, G \right\rangle_{A^2\left(\mathbb{C}_+, \, (\nu_n)_{n=0}^m\right)} := \sum_{n=0}^m \left\langle F^{(n)}, \, G^{(n)} \right\rangle_{A^2_{\nu_n}} \; \; \; \; \; \forall F, G \in A^2\left(\mathbb{C}_+, \, (\tilde{\nu}_n)_{n=0}^m\right).
\end{displaymath}
and as a consequence of the previous theorem, the Laplace transform defines an isometric map $\mathfrak{L} : L^2_{w_{(m)}} (0, \infty) \longrightarrow A^2\left(\mathbb{C}_+, \, (\nu_n)_{n=0}^m\right)$, where
\begin{displaymath}
	w_{(m)}(t) := \sum_{n=0}^m w_n(t) \quad \text{ and } \quad w_n(t) := 2\pi t^{2n} \int_0^\infty e^{-2rt} \, d\tilde{\nu}_n (r) \; \; \; \; \; (t > 0).
\label{eq:wm}
\end{displaymath}
We shall thereby restrict our attention to
\begin{displaymath}
	A^2_{(m)} := \mathfrak{L}\left(L^2_{w_{(m)}}(0, \infty)\right) \subseteq A^2\left(\mathbb{C}_+, \, (\nu_n)_{n=0}^m\right),
\end{displaymath}
where the inclusion becomes equality if and only if $\mathfrak{L}$ is a surjective map. It is the case, for example, when $\tilde{\nu}_0 = \delta_0$ (the Dirac delta function at 0), by Paley-Wiener Theorem \cite{duren1970}, \cite{hoffman1962}, \cite{mashreghi2009}, or $d\tilde{\nu}_0(r)=r^\alpha dr, (\alpha>-1)$ \cite{duren2007}, \cite{gallardo2009}. The surjectivity of $\mathfrak{L} : L^2_w (0, \infty) \longrightarrow A^2_\nu$ is discussed in \cite{zen2009}. $A^2_{(m)}$ is a closed subspace of $A^2\left(\mathbb{C}_+, \, (\nu_n)_{n=0}^m\right)$, and hence a Hilbert space on its own right. It is in fact a reproducing kernel Hilbert space, and  we can easily find its kernels. Given $F= \mathfrak{L}[f] \in A^2_{(m)}$, we have
\begin{align*}
F(z) &= \int_0^\infty f(t) \frac{e^{-tz}}{w_{(m)}(t)} w_{(m)}(t) \, dt \\
&= \left\langle f(t), \, \frac{e^{-t\overline{z}}}{w_{(m)}(t)} \right\rangle_{L^2_{w_{(m)}}(0, \infty)} \\
&\!\!\!\! \stackrel{\thm\ref{thm:mainthm}}{=} \left\langle F(\zeta), \mathfrak{L}\left[\frac{e^{-t\overline{z}}}{w_{(m)}(t)}\right](\zeta) \right\rangle_{A^2_{(m)}}.
\end{align*}
So, by the uniqueness of reproducing kernels \cite{aronszajn1950}, \cite{paulsen2009}, we have that the reproducing kernel of $A^2_{(m)}$ at $z \in \mathbb{C}_+$ is given by
\begin{equation}
k^{A^2_{(m)}}_z (\zeta) : = \int_0^\infty \frac{e^{-t(\zeta+\overline{z})}}{w_{(m)}(t)} \, dt.
\label{eq:kernel}
\end{equation}
This kernel coincides with the reproducing kernel of $A^2(\mathbb{C}_+, \, (\nu_n)_{n=0}^m)$ if and only if $\mathfrak{L}$ is surjective, which extends the Paley-Wiener theorem to more general spaces of analytic functions defined on the half-plane, and it is clear that $\mathfrak{L}$ is onto in case of Hardy or weighted Bergman spaces, and consequently in case of any space contained within them (in subset sense), by Theorem \ref{thm:mainthm}.

\section{Multipliers}
\label{sec:mult}
One of the main advantages of introducing the notion of the space $A^2_{(m)}$ is that it is a generalisation of the Dirichlet space on the right complex half-plane (amongst many others). Recall that a function $F$ is said to be in the Dirichlet space $\mathcal{D}(\mathbb{C}_+)$ if
\begin{displaymath}
	\left\|F\right\|^2_{\mathcal{D}(\mathbb{C}_+)} : = \left\|F\right\|^2_{H^2(\mathbb{C}_+)} + \int_{\mathbb{C}_+} \left|F'(z)\right|^2 \, dz < \infty.
\end{displaymath}
So if $\tilde{\nu}_0 = \frac{1}{2\pi} \delta_0$ (Dirac delta at 0) and $\tilde{\nu}_1$ is the Lebesgue measure (with weight $1/\pi$), then indeed $A^2_{(1)} = \mathcal{D}(\mathbb{C}_+)$. Therefore we may often adopt tools used in the study of the latter space (usually defined on the unit disk of the complex plane though), to study more general examples.

We define
\begin{displaymath}
	\mathscr{M}(A^2_{(m)}) : = \left\{ h: \mathbb{C}_+ \longrightarrow \mathbb{C} \; : \; hF \in A^2_{(m)}, \; \forall F \in A^2_{(m)} \right\},
\end{displaymath}
that is the \emph{space of multipliers of $A^2_{(m)}$}. Clearly if $h \in \mathscr{M}(A^2_{(m)})$, then it must be analytic (since $A^2_{(m)} \neq \{0\}$). For each $h \in \mathscr{M}(A^2_{(m)})$ we can also define the \emph{multiplication operator}, $M_h \in \mathscr{B}(A^2_{(m)})$ (i.e. the Banach algebra of bounded linear operators on $A^2_{(m)}$), by $M_h F := hF, \, F \in A^2_{(m)}$. We may associate the space of multipliers with the space of multiplication operators, equipping it with the \emph{multiplier norm}
\begin{displaymath}
	\left\|h\right\|_{\mathscr{M}(A^2_{(m)})} : = \left\|M_h\right\|_{\mathscr{B}(A^2_{(m)})} = \sup_{\left\|F\right\|_{A^2_{(m)}} \leq 1} \left\|h F\right\|_{A^2_{(m)}}.
\end{displaymath}

The following two lemmata are well-known and evidently hold for any reproducing kernel Hilbert space (see for example Theorems 5.1.4 and 5.1.5 from \cite{el-fallah2014} in case of $\mathcal{D}$).
\begin{lemma}
Let $M_h$ be a multiplication operator on $A^2_{(m)}$. Then
\begin{displaymath}
\left(M^*_h k^{A^2_{(m)}}_z\right)(\zeta) = \overline{h(z)} k^{A^2_{(m)}}_z(\zeta).
\end{displaymath}
\end{lemma}

\begin{proof}
Let $F \in A^2_{(m)}$.
\begin{displaymath}
\left\langle F, \, M^*_h( k^{A^2_{(m)}}_z) \right\rangle = \left\langle M_h(F), \,  k^{A^2_{(m)}}_z \right\rangle = h(z)F(z) = \left\langle F, \, \overline{h(z)}  k^{A^2_{(m)}}_z \right\rangle,
\end{displaymath}
and since it holds for all $F$ in $A^2_{(m)}$, we can deduce the desired result.
\end{proof}

\begin{lemma}
If $h \in \mathscr{M}(A^2_{(m)})$, then $h$ is bounded and $\left\| h \right\|_{H^\infty} \leq \left\| h \right\|_{\mathscr{M}(A^2_{(m)})}$.
\end{lemma}

\begin{proof}
Let $h \in \mathscr{M}(A^2_{(m)})$. Then $M^*_h$ is a bounded operator on $A^2_{(m)}$, so its eigenvalues are bounded, and of modulus no bigger than $\left\| M_h \right\|_{\mathscr{B}(A^2_{(m)})}$. By the previous lemma it follows that the values of $h$ are bounded and of modulus no more than $\left\| h \right\|_{\mathscr{M}(A^2_{(m)})}$.
\end{proof}

These results used to prove the following theorem, characterising the multipliers of $A^2_{(m)}$. Versions of this theorem for Hardy and Bergman spaces are obvious (see for example Proposition 1.13 in \cite{alsaker2009}) and can easily be extended to all Zen spaces. The version for the Dirichlet space is given in \cite{el-fallah2014} (Theorem 5.1.7), and describes multipliers using Carleson measures. Recall that a positive Borel measure $\mu$ on $\Omega \subseteq \mathbb{C}$ is called a \emph{Carleson measure} for a Hilbert space $\mathcal{H}$ of (complex-valued) functions defined on $\Omega$ if there exists a constant $C$ such that
\begin{displaymath}
	\int_\Omega \left|f\right|^2 \, d\mu \leq C \left\|f\right\|^2_{\mathcal{H}} \; \; \; \; \; (f \in \mathcal{H}).
\end{displaymath}

Let us now state and prove the characterisation of multipliers for the general case, $A^2_{(m)}$.

\begin{theorem} ~
\begin{enumerate}
	\item $\mathscr{M}(A^2_{(0)})=H^\infty(\mathbb{C}_+)$ and $\|h\|_{\mathscr{M}(A^2_{(0)})} = \|h\|_{H^\infty(\mathbb{C}_+)}$.
	\item If, for all $0 \leq k \leq n \leq m<\infty$, $\mu_{n,k}$, given by $d\mu_{n,k}(z) := \left|h^{(k)}\right|^2 d\nu_n$, is a Carleson measure for $A^2_{\nu_{n-k}}$, then $h \in \mathscr{M}(A^2_{(m)})$.
	\item If $1 \leq m < \infty$, then $h \in \mathscr{M}(A^2_{(m)})$ if and only if for all $F \in A^2_{(m)}$ and all $1\leq n \leq m$ there exists $c_n > 0$ such that
		\begin{equation}
			\int_{\mathbb{C}_+} \left| \sum_{k=1}^n \binom{n}{k} F^{(n-k)}h^{(k)} \right|^2 \, d\nu_n \leq c_n \left\|F\right\|^2_{A^2_{(m)}}.
			\label{eq:quasicarleson}
		\end{equation}
		In particular, if $m=1$, then $h \in \mathscr{M}(A^2_{(1)})$ if and only if $\left|h'(z)\right|^2 d\nu_1$ is a Carleson measure for $A^2_{(1)}$.
\end{enumerate}
\end{theorem}

\begin{proof}
The first part is obvious. For the second, let $F \in A^2_{(m)}$, then there exist constants $c_{n,k}$, for $0 \leq k \leq n \leq m$, such that
\begin{align*}
	\left\| Fh \right\|^2_{A^2_{(m)}} &= \sum_{n=0}^m \int_{\mathbb{C}_+} \left|(Fh)^{(n)}\right|^2 \, d\nu_n \\
	&\leq 2^m \sum_{n=0}^m \sum_{k=0}^n \binom{n}{k}^2 \int_{\mathbb{C}_+} \left|F^{(n-k)}h^{(k)}\right|^2 \, d\nu_n \\
	&\leq 2^m \sum_{n=0}^m \sum_{k=0}^n \binom{n}{k}^2 c_{n,k} \left\|F^{(n-k)}\right\|^2_{A^2_{\nu_{n-k}}} < \infty.
\end{align*}
For the last part suppose that \eqref{eq:quasicarleson} holds for some $h$. Then
\begin{align*}
\left\|Fh\right\|^2_{A^2_{(m)}} &= \sum_{n=0}^m \int_{\mathbb{C}_+} \left|(Fh)^{(n)}\right|^2 \, d\nu_n \\
&\leq \sum_{n=0}^m \int_{\mathbb{C}_+} \left|\sum_{k=0}^n \binom{n}{k} F^{(n-k)}h^{(k)}\right|^2 \, d\nu_n \\
&\stackrel{\eqref{eq:quasicarleson}}{\leq} 2\sum_{n=0}^m \left( \left\|h\right\|^2_{H^\infty(\mathbb{C}_+)} \int_{\mathbb{C}_+} \left|F^{(n)}\right|^2 \, d\nu_n + c_n \left\|F\right\|^2_{A^2_{(m)}} \right) \\
&= 2\left(\left\|h\right\|^2_{H^\infty(\mathbb{C}_+)} + \sum_{n=0}^m c_n \right) \left\|F\right\|^2_{A^2_{(m)}} < \infty,
\end{align*}
thus $h \in \mathscr{M}(A^2_{(m)})$. Conversely, suppose that $h \in \mathscr{M}(A^2_{(m)})$. Then
\begin{align*}
\int_{\mathbb{C}_+} \left| \sum_{k=1}^n \binom{n}{k} F^{(n-k)}h^{(k)} \right|^2 \, d\nu_n &= \int_{\mathbb{C}_+} \left| (Fh)^{(n)} - F^{(n)}h\right|^2 \, d\nu_n \\
& \leq 2 \left(\left\|Fh\right\|^2_{A^2_{(m)}} + \left\|h\right\|^2_{H^\infty(\mathbb{C}_+)} \left\|F\right\|^2_{A^2_{(m)}}\right) \\
&\leq 2 \left(\left\|M_h\right\|^2 + \left\|h\right\|^2_{H^\infty(\mathbb{C}_+)}\right) \left\|F\right\|^2_{A^2_{(m)}}.
\end{align*}
\end{proof}

\section{Banach algebras}
\label{sec:ba}

In an analogous way as in the previous section we can define the space of multipliers $\mathscr{M}(\mathcal{H})$ for an arbitrary Hilbert space of analytic functions $\mathcal{H}$. If $\mathcal{H}$ is a reproducing kernel Hilbert space, then $\mathscr{M}(\mathcal{H})$ is a unital Banach subalgebra of $\mathscr{B}(\mathcal{H})$, which is closed in the weak operator topology \cite{paulsen2009}. It is clear that unlike $\mathscr{M}(\mathcal{D}) \subset \mathcal{D}$, $\mathscr{M}(A^2_{(m)})$ is never a subset of $A^2_{(m)}$ (since constant functions are always in $\mathscr{M}(\mathcal{H})$, but can never be in $A^2_{(m)}$). So we may ask a question: is it possible to have the reverse inclusion, i.e. $A^2_{(m)} \subset \mathscr{M}(A^2_{(m)})$? And if so, what criteria need to be satisfied? It turns out that it is possible. We can choose measures $\tilde{\nu}_0, \ldots, \tilde{\nu}_m$ such that $A^2_{(m)}$ is a Banach algebra and hence, being closed under multiplication, it must be a (proper) subset $\mathscr{M}(A^2_{(m)})$. It was said above that $A^2_{(m)}$ is a reproducing kernel Hilbert space. We now want to find $(\nu_n)_{n=0}^m$ such that $A^2_{(m)}$ is also a Banach algebra. Let us start with the following observation.

\begin{theorem}
Let $\mathcal{H}$ be a Hilbert space of complex-valued functions defined on a domain $\Omega \subseteq \mathbb{C}$, which is also a Banach algebra with respect to pointwise multiplication. Then $\mathcal{H}$ is a reproducing kernel Hilbert space, and if $k_z$ is the reproducing kernel of $\mathcal{H}$ at $z \in \Omega$, then
\begin{equation}
	\sup_{z \in \Omega} \left\|k_z\right\|_{\mathcal{H}} \leq 1,
	\label{eq:ker}
\end{equation}
and consequently all elements of $\mathcal{H}$ are bounded.
\end{theorem}

\begin{proof}
First, note that the evaluation functional $E_\lambda : \mathcal{H} \longrightarrow \mathbb{C}, \, f \stackrel{E_\lambda}{\mapsto} f(\lambda)$ is bounded for every $\lambda \in \Omega$, since it is a multiplicative functional on a Banach algebra $\mathcal{H}$ and hence $\left\|E_\lambda\right\| \leq 1$ (see \cite{bonsall1973}, \S16, Proposition 3, p. 77), so $\mathcal{H}$ is a reproducing kernel Hilbert space (see \cite{aronszajn1950}, \cite{paulsen2009}). Let $k_z$ denote the reproducing kernel of $\mathcal{H}$ at $z \in \Omega$. Then we have
\begin{equation}
	\left\|k_z\right\|^2_\mathcal{H} = \left|k_z(z)\right| \leq \sup_{\zeta \in \Omega} \left|k_z(\zeta)\right|.
\label{eq:bak1}
\end{equation}
Also, by the Cauchy-Schwarz inequality and the fact that $\mathcal{H}$ is a Banach algebra, we get
\begin{align*}
 \left|k_z(\zeta) \right| \left\|k_\zeta \right\|^2_\mathcal{H} &= \left|k_z(\zeta)k_\zeta(\zeta)\right| \\
&= \left|\left\langle k_z k_\zeta, \, k_\zeta \right\rangle\right| \\
&\leq \left\|k_z k_\zeta \right\|_\mathcal{H} \left\|k_\zeta\right\|_\mathcal{H} \\
&\leq \left\|k_z \right\|_\mathcal{H} \left\|k_\zeta \right\|^2_\mathcal{H},
\end{align*}
and since it holds for all $z, \zeta \in \Omega$, after canceling $\left\|k_\zeta \right\|^2_\mathcal{H}$ and taking the supremum, we get
\begin{equation}
	\sup_{\zeta \in \Omega} \left|k_z(\zeta)\right|  \leq \left\|k_z \right\|_\mathcal{H}.
\label{eq:bak2}
\end{equation}
From \eqref{eq:bak1} and \eqref{eq:bak2} we get
\begin{displaymath}
\left\|k_z \right\|^2_\mathcal{H} \leq  \sup_{\zeta \in \Omega} \left|k_z(\zeta)\right| \leq \left\|k_z \right\|_\mathcal{H}
\end{displaymath}
and consequently
\begin{displaymath}
\left\|k_z\right\|_\mathcal{H} \leq 1.
\end{displaymath}
And for any $f \in \mathcal{H}$ we also have
\begin{displaymath}
	\sup_{z \in \Omega} \left|f(z)\right| = \sup_{z \in \Omega} \left|\left\langle f, \, k_z \right\rangle\right| \leq \left\|f\right\|_\mathcal{H}.
\end{displaymath}
\end{proof}

\begin{theorem}
\label{thm:ban}
If $A^2_{(m)}$ is a Banach algebra then
\begin{displaymath}
	\int_0^\infty \frac{dt}{w_{(m)}(t)} \leq 1,
\end{displaymath}
and therefore
\begin{displaymath}
	L^2_{w_{(m)}}(0, \infty) \subseteq L^1(0, \infty) \; \; \; \; \; \text{ and } \; \; \; \; \; A^2_{(m)} \subseteq \mathscr{M}(A^2_{(m)}) \cap H^\infty(\mathbb{C}_+) \cap \mathcal{C}_0(i\mathbb{R}).
\end{displaymath}
Conversely, if for all $t>0$
\begin{equation}
	\left(\frac{1}{w_{(m)}} \ast \frac{1}{w_{(m)}}\right)(t) \leq \frac{1}{w_{(m)}(t)},
\label{eq:convo}
\end{equation}
then $A^2_{(m)}$ is a Banach algebra.
\end{theorem}

\begin{proof}
Suppose that ${A^2_{(m)}}$ is a Banach algebra, then by the previous theorem
\begin{equation}
\int_0^\infty \frac{1}{w_{(m)}(t)} \, dt = \sup_{z \in \mathbb{C}_+}  \int_0^\infty \frac{e^{-2\re(z)t}}{w_{(m)}(t)} \, dt \stackrel{\eqref{eq:kernel}}{=} \sup_{\zeta \in \mathbb{C}_+} \left\|k^{A^2_{(m)}}_\zeta \right\|_{A^2_{(m)}} \stackrel{\eqref{eq:ker}}{\leq} 1.
\label{eq:bak5}
\end{equation}
By H\"{o}lder's inequality we also get
\begin{displaymath}
\left| \int_0^\infty f(t) e^{-tz} \, dt \right| \leq \int_0^\infty \left|f(t)\right| \, dt \stackrel{\eqref{eq:bak5}}{\leq} \left(\int_0^\infty \left|f(t)\right|^2 w_{(m)}(t) \, dt\right)^{\frac{1}{2}},
\end{displaymath}
and on the boundary
\begin{displaymath}
	F(\im(z)) = \int_0^\infty f(t)e^{-i\im(z)t} \, dt \in \mathcal{C}_0(i\mathbb{R}).
\end{displaymath}
The converse follows from the fact that multiplication in $A^2_{(m)}$ is equivalent to convolution in $L^2_{w_{(m)}}(0, \infty)$, for which the sufficient condition to be a Banach algebra was given in \cite{nikolskii1970} and in \cite{acta1995} (Lemma 8.11) and its proof is quoted here. Suppose that \eqref{eq:convo} holds for all $t>0$. Using H\"{o}lder's inequality and that $(L^1(0, \infty), \ast)$ is a Banach algebra \cite{dales2000}, we get
\begin{align*}
\left\|f \ast g \right\|^2_{L^2_{w_{(m)}}(0, \infty)} &= \int_0^\infty \left| \int_0^t f(\tau)g(t-\tau) \, d\tau\right|^2 w_{(m)}(t) \, dt \\
&\leq \int_0^\infty \int_0^t \left|f(\tau)\right|^2 w_{(m)}(\tau) \left|g(t-\tau)\right|^2 w_{(m)}(t-\tau) \, d\tau \\
& \; \; \; \; \; \; \times \int_0^t \frac{d\tau}{w_{(m)}(\tau) w_{(m)}(t-\tau)} \, w_{(m)}(t) \, dt \\
&= \int_0^\infty (\left|f\right|^2 w_{(m)} \ast \left|g\right|^2 w_{(m)})(t) \, \left( \frac{1}{w_{(m)}} \ast \frac{1}{w_{(m)}} \right)(t) \, w_{(m)}(t) \, dt \\
&\stackrel{\eqref{eq:convo}}{\leq} \left\| \left|f\right|^2 w_{(m)} \right\|_{L^1(0, \infty)} \left\| \left|g\right|^2 w_{(m)} \right\|_{L^1(0, \infty)} \\
&=\left\|f\right\|^2_{L^2_{w_{(m)}}(0, \infty)} \left\|g\right\|^2_{L^2_{w_{(m)}}(0, \infty)}
\end{align*}
for all $f, g$ in $L^2_{w_{(m)}}(0, \infty)$, and hence $A^2_{(m)}$ is a Banach algebra.
\end{proof}

\begin{example}
$H^2(\mathbb{C}_+), \, \mathcal{B}^2_\alpha (\mathbb{C}_+)$ are not Banach algebras. In fact, no Zen space can be a Banach algebra, since
\begin{displaymath}
	w_0(t) \defeq 2\pi \int_0^\infty e^{-2rt} \, d\tilde{\nu}_0 (r)
\end{displaymath}
is a decreasing function. $\mathcal{D} (\mathbb{C}_+)$ is not a Banach algebra either. The above necessary condition is a good tool in disqualifying given $A^2_{(m)}$ from being a Banach algebra. The sufficient condition is somehow less useful in producing examples of Banach algebras and it was not even clear if they existed. They do, and there is an alternative way to produce them.
\end{example}

\begin{theorem}
\label{thm:bal} ~
\begin{enumerate}
	\item $A^2_{\nu} \cap H^\infty(\mathbb{C}_+)$ is a Banach algebra with norm given by
\begin{displaymath}
	\left\|F\right\|_{A^2_\nu \cap H^\infty(\mathbb{C}_+)} := \left\| F \right\|_{H^\infty(\mathbb{C}_+)} + \left\|F\right\|_{A^2_\nu} \; \; \; \; \; (\forall F \in A^2_{\nu} \cap H^\infty(\mathbb{C}_+)).
\end{displaymath}
	\item 
	Suppose that for all $1 \leq k < n \leq m - 1 < \infty$ and all $t \geq 0$ we have
	\begin{equation}
		\int_0^\infty e^{-2rt} \, d\tilde{\nu}_n (r) \leq \int_0^\infty e^{-2rt} \, d\tilde{\nu}_{n-k} (r),
	\label{eq:measure}
	\end{equation}
then
	\begin{displaymath}
		\Alg_m := \bigcap_{n=0}^{m-1} \left\{ F \in A^2_{(m)} \, : \, F^{(n)} \in H^\infty(\mathbb{C}_+) \right\}
	\end{displaymath}
is a Banach algebra with respect to the norm given by
	\begin{displaymath}
		\left\|F \right\|_{\Alg_m} : = \sum_{n=0}^{m-1} \frac{\left\|F^{(n)}\right\|_{H^\infty(\mathbb{C}_+)}}{n!} + \sum_{n=0}^{m} \frac{\left\|F^{(n)}\right\|_{A^2_{\nu_n}}}{n!}.
	\end{displaymath}
\end{enumerate}

\end{theorem}

\begin{proof}
Those are clearly Banach spaces. For all $F$ and $G$ in $A^2_\nu \cap H^\infty(\mathbb{C}_+)$
\begin{align*}
\left\|FG\right\|_{A^2_\nu \cap H^\infty(\mathbb{C}_+)} &\! \defeq  \left\| FG \right\|_{H^\infty(\mathbb{C}_+)} + \left\|FG\right\|_{A^2_\nu} \\
& \leq \left\| F \right\|_{H^\infty(\mathbb{C}_+)} \left\| G \right\|_{H^\infty(\mathbb{C}_+)} + \left\| F \right\|_{H^\infty(\mathbb{C}_+)} \left\|G\right\|_{A^2_\nu} \\
& \leq \left(\left\| F \right\|_{H^\infty(\mathbb{C}_+)} + \left\|F\right\|_{A^2_\nu} \right)\left(\left\| G \right\|_{H^\infty(\mathbb{C}_+)} + \left\|G\right\|_{A^2_\nu}\right) \\
&\! \defeq \left\|F\right\|_{A^2_\nu \cap H^\infty(\mathbb{C}_+)} \left\|G\right\|_{A^2_\nu \cap H^\infty(\mathbb{C}_+)},
\end{align*}
proving 1. To prove 2., let $F$ and $G$ be in $\Alg_m$, and let
\begin{align*}
f_n &= \frac{\left\|F^{(n)}\right\|_{H^\infty(\mathbb{C}_+)}}{n!} \; \; \; \text{for } 0 \leq n < m \; \; \; \text{ and } \; \; f_m = 0, \\
f'_n &= \frac{\left\|F^{(n)}\right\|_{A^2_{\nu_n}}}{n!}, \\
g_n &= \frac{\left\|G^{(n)}\right\|_{H^\infty(\mathbb{C}_+)}}{n!} \; \; \; \text{for } 0 \leq n < m \; \; \; \text{ and } \; \; g_m = 0, \\
g'_n &= \frac{\left\|G^{(n)}\right\|_{A^2_{\nu_n}}}{n!}.
\end{align*}
Then \eqref{eq:measure} implies
\begin{displaymath}
\int_{\mathbb{C}_+} \left|F^{(n-k)}\right|^2 \, d\nu_n \leq \int_{\mathbb{C}_+} \left|F^{(n-k)}\right|^2 \, d\nu_{n-k},	
\end{displaymath}
and then
\begin{align*}
\left\|FG \right\|_{\Alg_m} &\!\defeq \sum_{n=0}^{m-1} \frac{\left\|(FG)^{(n)}\right\|_{H^\infty(\mathbb{C}_+)}}{n!} + \sum_{n=0}^{m} \frac{\left\|(FG)^{(n)}\right\|_{A^2_{\nu_n}}}{n!} \\
&\leq \sum_{n=0}^{m-1} \frac{1}{n!} \sum_{k=0}^n \binom{n}{k} \left\|F^{(n-k)}\right\|_{H^\infty(\mathbb{C}_+)} \left\|G^{(k)}\right\|_{H^\infty(\mathbb{C}_+)} \\
&+ \sum_{n=0}^m \frac{1}{n!} \sum_{k=0}^n \binom{n}{k} \left(\int_{\mathbb{C}_+} \left|F^{(n-k)} G^{(k)} \right|^2 \, d\nu_n \right)^{1/2} \\
&\leq \sum_{n=0}^{m-1} \sum_{k=0}^n \frac{\left\|F^{(n-k)}\right\|_{H^\infty(\mathbb{C}_+)}}{(n-k)!} \frac{\left\|G^{(k)}\right\|_{H^\infty(\mathbb{C}_+)}}{k!} + \left\|F\right\|_{A^2_{\nu_0}} \left\|G\right\|_{H^\infty(\mathbb{C}_+)} \\
& \; \; \; \; \; + \sum_{n=1}^m \frac{1}{n!} \sum_{k=0}^{n-1} \binom{n}{k} \left\|F^{(n-k)}\right\|_{A^2_{\nu_{n-k}}} \left\|G^{(k)}\right\|_{H^\infty(\mathbb{C}_+)} \\
& \; \; \; \; \; + \sum_{n=1}^m \frac{1}{n!} \left\|F\right\|_{H^\infty(\mathbb{C}_+)} \left\|G^{(n)}\right\|_{A^2_{\nu_n}} \\
&= \sum_{n=0}^{m-1} \sum_{k=0}^n f_{n-k}g_k + f'_0 g_0 + \sum_{n=1}^{m} \sum_{k=0}^{n-1} f'_{n-k}g_k + f_0 \sum_{n=1}^{m} g'_n \\
&\leq \sum_{n=0}^m \sum_{k=0}^n (f_{n-k}g_k +  f'_{n-k}g_k +  f_{n-k}g'_k +  f'_{n-k}g'_k) \\
&= \left[\sum_{n=0}^m \left(f_n + f'_n\right)\right]\left[\sum_{n=0}^m \left(g_n + g'_n\right)\right] \\
&\! \defeq \left\|F\right\|_{\Alg_m} \left\|G\right\|_{\Alg_m},
\end{align*}
as required.
\end{proof}

\begin{theorem}
\label{thm:bana}
Let $m \in \mathbb{N}$ and let $(\nu_n)_{n=0}^m$ be a finite sequence of positive regular Borel measures. Suppose that for all $1 \leq k < n \leq m - 1$ and all $t \geq 0$ \eqref{eq:measure} holds (up to a constant). If
\begin{equation}
\int_0^\infty \frac{dt}{w_{m-1}(t)+w_m(t)} \leq 1
\label{eq:wbounded}
\end{equation}
then there exists a constant $C>0$ such that $\left(A^2_{(m)}, C\left\|\cdot\right\|_{A^2_{(m)}}\right)$ is a Banach algebra.
\end{theorem}

\begin{proof}
Given $0 \leq n \leq m$, let
	\begin{displaymath}
		B^2_{(m-n)} = \left\{ G: \mathbb{C}_+ \longrightarrow \mathbb{C} \, \text{analytic : }\left\|G\right\|^2_{B^2_{(m-n)}} = \sum_{k=0}^{m-n} \int_{\mathbb{C}_+} \left|G^{(k)} \right|^2 d\nu_{n+k} < \infty \right\},
	\end{displaymath}
(so it is a truncated $A^2_{(m)}$ space, with first $n$ measures removed). Note that if $F \in A^2_{(m)}$, then $F^{(n)}$ is in $B^2_{(m-n)}$ and for all $z \in \mathbb{C}_+$
\begin{displaymath}
\left|F^{(n)}(z)\right|^2 = \left|\left\langle F^{(n)}, \, k_z^{B^2_{(m-n)}} \right\rangle\right|^2 \stackrel{\eqref{eq:kernel}}{\leq} \left\|F^{(n)}\right\|^2_{B^2_{(m-n)}} \int_0^\infty \frac{e^{-2t\re(z)}}{w_n(t)+ \ldots + w_m(t)} \, dt,
\end{displaymath}
so clearly
\begin{displaymath}
\left\|F^{(n)}\right\|^2_{H^\infty(\mathbb{C}_+)} \leq \left\|F^{(n)}\right\|^2_{B^2_{(m-n)}} \int_0^\infty \frac{e^{-2t\re(z)}}{w_{m-1}(t) + w_m(t)} \, dt \stackrel{\eqref{eq:wbounded}}{\lessapprox} \left\|F\right\|^2_{A^2_{(m)}}
\end{displaymath}
for all $0 \leq n \leq m-1$. Now let $F, \, G \in A^2_{(m)}$ and let $d\nu_n':= n!d\nu_n$. Then for any $0 \leq k \leq m$
\begin{align*}
\left\|(FG)^{(k)}\right\|^2_{A^2_{\nu_k}} &\leq \sum_{n=1}^{m-1} \frac{\|(FG)^{(n)}\|_{H^\infty(\mathbb{C}_+)}}{n!}+\sum_{n=1}^m \frac{\|(FG)^{(n)}\|_{A^2_{\nu_n'}}}{n!} \\
&\!\!\!\! \stackrel{\thm\ref{thm:bal}}{\leq} \left(\sum_{n=1}^{m-1} \frac{\|F^{(n)}\|_{H^\infty(\mathbb{C}_+)}}{n!}+\sum_{n=1}^m \frac{\|F^{(n)}\|_{A^2_{\nu_n'}}}{n!}\right) \\
&\; \; \, \, \times \left(\sum_{n=1}^{m-1} \frac{\|G^{(n)}\|_{H^\infty(\mathbb{C}_+)}}{n!}+\sum_{n=1}^m \frac{\|G^{(n)}\|_{A^2_{\nu_n'}}}{n!}\right) \\
&\lessapprox \left(\sum_{n=1}^m \|F^{(n)}\|_{A^2_{\nu_n}}\right) \left(\sum_{n=1}^m \|G^{(n)}\|_{A^2_{\nu_n}}\right) \\
&\lessapprox \|F\|_{A^2_{(m)}} \|G\|_{A^2_{(m)}},
\end{align*}
summing the square of the above expression over all $k$ between 0 and $m$ and taking the square roots proves the claim. Or, to be precise, by multiplying the weights by appropriate constants, we can assure that $A^2_{(m)}$ is a Banach algebra.
\end{proof}

\begin{corollary}
$A^2_{(1)}$ is a Banach algebra (after possibly adjusting its norm/weights) if and only if
\begin{displaymath}
	\int_0^\infty \frac{dt}{w_{(1)}(t)} \leq \infty.
\end{displaymath}
\end{corollary}

\begin{proof}
It follows from Theorems \ref{thm:ban} and \ref{thm:bana}.
\end{proof}

\begin{example}
If $\tilde{\nu}_0 = \tilde{\nu}_1 = \delta_0$, then $A^2_{(1)}$ (that is a Hardy-Sobolev space) is a Banach algebra, since
\begin{displaymath}
	\frac{1}{2\pi} \int_0^\infty \frac{dt}{1+t^2} = \frac{1}{4}.
\end{displaymath}
It is easy to see that no adjustment in norm is necessary, as for any $F, G \in A^2_{(1)}$ we have
\begin{align*}
\|FG\|^2_{A^2_{\nu_0}}+\|(FG)'\|^2_{A^2_{\nu_1}} &\leq \frac{\|F\|^2_{H^\infty(\mathbb{C}_+)}\|G\|^2_{A^2_{\nu_0}}+\|G\|^2_{H^\infty(\mathbb{C}_+)}\|F\|^2_{A^2_{\nu_0}}}{2} \\
&\quad + 2\|F\|^2_{H^\infty(\mathbb{C}_+)}\|G'\|2_{A^2_{\nu_1}}+2\|G\|^2_{H^\infty(\mathbb{C}_+)}\|F'\|^2_{A^2_{\nu_1}} \\
&\leq \|F\|^2_{H^\infty(\mathbb{C}_+)} \left(\frac{\|G\|^2_{A^2_{\nu_0}}}{2}+2\|G'\|2_{A^2_{\nu_1}}\right) \\
&\quad +\|G\|^2_{H^\infty(\mathbb{C}_+)} \left(\frac{\|F\|^2_{A^2_{\nu_0}}}{2}+2\|F'\|2_{A^2_{\nu_1}}\right) \\
&\leq \frac{1}{4}\|F\|^2_{A^2_{(1)}} \cdot 2 \|G\|^2_{A^2_{(1)}}+\frac{1}{4}\|G\|^2_{A^2_{(1)}} \cdot 2 \|F\|^2_{A^2_{(1)}} \\
&=\|F\|^2_{A^2_{(1)}} \|G\|^2_{A^2_{(1)}}
\end{align*}
\end{example}

\begin{example}
If $\tilde{\nu}_0 = \delta_0$ and $d\tilde{\nu}_1(r) = r^\alpha dr \; (-1<\alpha<0)$, then $A^2_{(1)}$ is a Banach algebra.
\end{example}

\section{Spectra and ideals}
\label{sec:specideals}

Recall that the \emph{spectrum of an element $a$ of an algebra $A$ over $\mathbb{C}$} is the set
\begin{displaymath}
	\sigma(A,a) := \left\{\lambda \in \mathbb{C} \; : \; (a-\lambda)^{-1} \notin A \right\}
\end{displaymath}
if $A$ is unital, and
\begin{displaymath}
	\sigma(A,a) := \left\{0\right\} \cup \left\{\lambda \in \mathbb{C} \; : \; a+\lambda b -ab \neq 0, \; \forall b \in A \right\}.
\end{displaymath}
otherwise. The \emph{spectral radius, r(a), of a} is defined by
\begin{displaymath}
	r(a) : = \sup \left\{\lambda \in \sigma(A,a) \right\}.
\end{displaymath}
It is well known that
\begin{displaymath}
	\sup_{\varphi \in \mathfrak{M}(A)}\left|\phi(a)\right| = r(a),
\end{displaymath}
where $\mathfrak{M}(A)$ is the maximal ideal space of $A$, i.e. the set of algebra homomorphisms defined on $A$ (for details, see for example \cite{bonsall1973}).

\begin{theorem} ~
\begin{enumerate}
	\item If $h \in \mathscr{M}(A^2_{(m)})$, then 
	\begin{displaymath}
		\overline{h(\mathbb{C}_+)} \subseteq \sigma(\mathscr{M}(A^2_{(m)}), h),
	\end{displaymath}
	with equality at least for $m~\leq~1$.
	\item If $F \in A^2_{(m)}$, then $F^{-1} \notin A^2_{(m)}$.
\end{enumerate}
\end{theorem}

\begin{proof}
Let $h \in \mathscr{M}(A^2_{(m)})$. We have that $(h-\lambda)^{-1} \in H^\infty(\mathbb{C}_+)$, for some $\lambda \in \mathbb{C}$, if and only if $\inf_{z \in \mathbb{C}_+} \left|h(z)-\lambda \right| > 0$, and consequently $\sigma(H^\infty(\mathbb{C}_+),h) = \overline{h(\mathbb{C}_+)}$. If $\lambda \in \sigma(H^\infty(\mathbb{C}_+),h)$, then $(h-\lambda)^{-1} \notin H^\infty(\mathbb{C}_+) \supseteq \mathscr{M}(A^2_{(m)})$, so clearly $\overline{h(\mathbb{C}_+)} = \sigma(H^\infty(\mathbb{C}_+),h) \subseteq \sigma(\mathscr{M}(A^2_{(m)}), h)$. For the reverse inclusion recall that $\mathscr{M}(A^2_{(0)}) = H^\infty(\mathbb{C}_+)$, and also observe that if $h^{-1} \in H^\infty(\mathbb{C}_+)$, then
\begin{align*}
	\int_{\mathbb{C}_+} \left|\left(\frac{F}{h}\right)'\right|^2 \, d\nu_1 &\lessapprox \left( \left\|\frac{1}{h}\right\|^2_{H^\infty(\mathbb{C}_+)} \int_{\mathbb{C}+}\left|h'F\right|^2 \, d\nu_1 + \left\|h\right\|^2_{H^\infty(\mathbb{C}_+)}\right) \int_{\mathbb{C}+}\left|F'\right|^2 \, d\nu_1  \\
	&\stackrel{\eqref{eq:quasicarleson}}{\lessapprox} \left(\left\|\frac{1}{h}\right\|^2_{H^\infty(\mathbb{C}_+)} + \left\|h\right\|^2_{H^\infty(\mathbb{C}_+)}\right)\left\|F\right\|^2_{A^2_{(1)}} < \infty.
\end{align*}
That is $h^{-1} \in \mathscr{M}(A^2_{(1)})$.
For the second part, note that for every $F \in A^2_{(m)}$ and each $\varepsilon > 0$ there exists a domain $\Omega_\varepsilon \subseteq \mathbb{C}_+$ with infinite measure, such that $\left|F(z)\right| < \varepsilon$, for all $z \in \Omega_\varepsilon$, but then $\left|1/F(z)\right| > \varepsilon$ on the same region, so it cannot be in $A^2_{(m)}$.
\end{proof}

For the remaining part of this paper we shall assume that $A^2_{(m)}$ is a Banach algebra.
\begin{remark}
If $A^2_{(m)}$ is an algebra, then it must, by definition, be an ideal in $\mathscr{M}(A^2_{(m)})$. A natural question to ask here is: whether it could be a maximal ideal. If that was the case, then it would be the kernel of an algebra homomorphism, so $\codim_{\mathscr{M}(A^2_{(m)})} A^2_{(m)} = 1$, and since $A^2_{(m)} \subset A^2_{(m)} + \mathbb{C} \subseteq \mathscr{M}(A^2_{(m)})$, we must have $\mathscr{M}(A^2_{(m)}) = A^2_{(m)} + \mathbb{C}$. Conversely $A^2_{(m)} + \mathbb{C}$ is the canonical unitisation of $A^2_{(m)}$, so if $\mathscr{M}(A^2_{(m)}) = A^2_{(m)} + \mathbb{C}$, then $A^2_{(m)}$ must be a maximal ideal there. In the simplest case, when $m=1$ and $\nu_0=\nu_1$ we have that
\begin{displaymath}
	\left\{g \in H^\infty(\mathbb{C}_+) \; : \; g' \in H^\infty(\mathbb{C}_+) \right\} \subseteq \mathscr{M}(A^2_{(1)}),
\end{displaymath}
so $A^2_{(1)}$ cannot be a maximal ideal (take for example $e^{-z} \notin A^2_{(m)} + \mathbb{C}$, which is obviously a multiplier). It still remains unclear whether there exists a sequence of measures $(\nu_n)_{n=0}^m$ such that $A^2_{(m)}$ would be a maximal ideal in the space of its multipliers.
\end{remark}

\begin{theorem}
Suppose that for each $a > 0$ there exists $K>0$ such that $w_{(m)}(t) \leq Ke^{at}$, for all $t >0$. Let $\pi : \mathfrak{M}(A^2_{(m)}) \longrightarrow \overline{\mathbb{D}}$ (the closed unit disk of the complex plane) be given by
\begin{displaymath}
	\pi(\varphi) = \varphi\left(\frac{1-z}{1+z}\right) \; \; \; \; \; (\varphi \in \mathfrak{M}(\mathscr{M}(A^2_{(m)})))
\end{displaymath}
(that is, $\pi$ is the Gel'fand transform of the function $(1-z)/(1+z)$. Then
\begin{enumerate}
	\item $\pi$ is surjective.
	\item If $m=1$ or $(\tilde{\nu}_n)_{n=0}^m$ satisfy \eqref{eq:measure}, then  $\pi$ is injective over the open unit disk $\mathbb{D}$ and $(\pi|^{\mathbb{D}})^{-1}$ (that is, the inverse of the restriction of $\pi$ to $\mathbb{D}$ in its image) maps $\mathbb{D}$ homeomorphically onto an open subset $\Delta \subset \mathfrak{M}(\mathscr{M}(A^2_{(m)}))$.
\end{enumerate}
\end{theorem}

\begin{proof}
First, note that, for all $\alpha \in \mathbb{C}_+$, $\frac{1}{z+\alpha}$ is in $A^2_{(m)}$, since
\begin{displaymath}
\int_0^\infty \left|e^{-\alpha t}\right|^2 w_{(m)}(t) \, dt \lessapprox \int_0^\infty e^{-t\re(\alpha)} \, dt = \frac{1}{\re(\alpha)}
\end{displaymath}
so $e^{-\alpha t} \in L^2_{w_{(m)}}(0, \infty)$, and hence
\begin{displaymath}
	\frac{1-z}{1+z} = \frac{2}{1+z} - 1 \in A^2_{(m)} + \mathbb{C} \subseteq \mathscr{M}(A^2_{(m)}).
\end{displaymath}
We know that $\sigma\left(\mathscr{M}(A^2_{(m)}), \frac{1-z}{1+z}\right) \supseteq \sigma\left(H^\infty(\mathbb{C}_+), \frac{1-z}{1+z}\right)$. And if $\left(\frac{1-z}{1+z} - \lambda\right)^{-1} \in H^\infty(\mathbb{C}+)$, for some $\lambda \in \mathbb{C}$, then
\begin{displaymath}
\left(\frac{1-z}{1+z} - \lambda\right)^{-1} = \frac{1+z}{1-\lambda-z(1+\lambda)} = \frac{1}{1+\lambda} \left[-\frac{2}{1+\lambda}\underbrace{\left(z-\frac{1-\lambda}{1+\lambda}\right)^{-1}}_{\in A^2_{(m)}} - 1 \right]
\end{displaymath}
is a multiplier on $A^2_{(m)}$. So we actually have 
\begin{displaymath}
\sigma\left(\mathscr{M}(A^2_{(m)}), \frac{1-z}{1+z}\right) = \sigma\left(H^\infty(\mathbb{C}_+), \frac{1-z}{1+z}\right)
\end{displaymath}
and hence
\begin{displaymath}
\left|\pi(\varphi)\right| \leq \sup_{\varphi \in \mathfrak{M}(\mathscr{M}(A^2_{(m)}))} \left|\varphi\left(\frac{1-z}{1+z}\right)\right| = r\left(\frac{1-z}{1+z}\right)=1.
\end{displaymath}
Since the evaluation homomorphisms are in $\mathfrak{M}(\mathscr{M}(A^2_{(m)}))$, every point of the open unit disk is in the image of $\pi$. Also, $\mathscr{M}(A^2_{(m)})$ is unital, and hence compact, so its image under $\pi$ must also be compact and thus $\pi$ is surjective.
For the second part, let $\left|\lambda\right| <1$ and suppose that $\pi(\varphi) = \lambda$. Then for any $F \in A^2_{(m)}$ vanishing at $\kappa = \frac{1-\lambda}{1+\lambda} \in \mathbb{C}_+$, we have $F= \frac{z-\kappa}{z+\overline{\kappa}}G$, with $G \in H^\infty(\mathbb{C}_+)$ (see \cite{mashreghi2009}, p. 293). Let $\overline{B_r(\kappa)}$ be the closed ball, centred at $\kappa$, with radius $r>0$. Choose $r$ small enough to get $\overline{B_r(\kappa)} \subset \mathbb{C}_+$, then
\begin{displaymath}
	\int_{\mathbb{C}_+} \left|G\right|^2 \, d\nu_0 = \int_{\overline{B_r(\kappa)}} \left|G\right|^2 \, d\nu_0  + \int_{\mathbb{C}_+ \setminus \overline{B_r(\kappa)}} \left|\frac{z+\overline{\kappa}}{z-\kappa}F\right|^2 \, d\nu_0.
\end{displaymath}
The first integral is finite, since $G$ is bounded and $\overline{B_r(\kappa)}$ is compact. The second one is also finite, since $\frac{z+\overline{\kappa}}{z-\kappa}$ is bounded on $\mathbb{C}_+ \setminus \overline{B_r(\kappa)}$. Let
\begin{displaymath}
	c := \sup_{z \in \mathbb{C}_+ \setminus \overline{B_r(\kappa)}} \left|\frac{z+\overline{\kappa}}{z-\kappa}\right|.
\end{displaymath}
Then we have
\begin{align*}
\int_{\mathbb{C}_+} \left|G'\right|^2 \, d\nu_1 &= \int_{\overline{B_r(\kappa)}} \left|G'\right|^2 \, d\nu_1 + \int_{\mathbb{C}_+ \setminus \overline{B_r(\kappa)}} \left|F'\frac{z+\overline{\kappa}}{z-\kappa} - F\frac{2\re{\kappa}}{(z-\kappa)^2}\right|^2 \, d\nu_1 \\
&\leq \int_{\overline{B_r(\kappa)}} \left|G'\right|^2 \, d\nu_1+ c^2 \left\|F'\right\|^2_{A^2_{\nu_1}} \\
&\; \; \; \; \; +4(c\re(\kappa))^2 \left\|F\right\|^2_{H^\infty(\mathbb{C}_+)} \int_{\mathbb{C}_+} \left|(z+\overline{\kappa})^{-2}\right| \, d\nu_1
\end{align*}
which is also finite, since $\left|G'\right|^2$ is continuous, $\overline{B_r(\kappa)}$ is compact and $(z+\overline{\kappa})^{-1} \in A^2_{(m)}$ implies $(z+\overline{\kappa})^{-2} \in A^2_{\nu_1}$. If $n > 1$, then
\begin{align*}
\int_{\mathbb{C}_+} \left|G^{(n)}\right|^2 \, d\nu_n &= \int_{\overline{B_r(\kappa)}} \left|G^{(n)}\right|^2 \, d\nu_n \\
&\; \; \; \; \; + \int_{\mathbb{C}_+ \setminus \overline{B_r(\kappa)}} \left|\sum_{k=0}^n \binom{n}{k} F^{(n-k)} \left(\frac{z+\overline{\kappa}}{z-\kappa}\right)^{(k)} \right|^2 \, d\nu_n \\
&\lessapprox \int_{\overline{B_r(\kappa)}} \left|G^{(n)}\right|^2 \, d\nu_n + \sum_{k=1}^{n-1} \left\|F^{(n-k)}\right\|^2_{A^2_{\nu_{n-k}}} \\
&\; \; \; \; \; + \left\|F\right\|^2_{H^\infty(\mathbb{C}_+)} \left\|(z+\overline{\kappa})^{-n}\right\|^2_{A^2_{\nu_n}} < \infty.
\end{align*}
Therefore $G \in A^2_{(m)}$. Let
\begin{displaymath}
	H := \underbrace{-\frac{(1+z)(1+\kappa)}{2(z+\overline{\kappa})}}_{\in \mathscr{M}(A^2_{(m)})} G \in A^2_{(m)}.
\end{displaymath}
Then
\begin{displaymath}
	\varphi(F) = \varphi\left(\frac{1-z}{1+z}-\lambda \right)\varphi(H) = 0.
\end{displaymath}
For any $h \in \mathscr{M}(A^2_{(m)})$, which vanishes at $\kappa$ we then have
\begin{displaymath}
0 = \varphi\left(\frac{h}{z+1}\right) = \varphi(h) \varphi\left(\frac{1}{1+z}\right) = \frac{\varphi(h)}{2} \varphi\left(\frac{1-z}{1+z} + 1 \right) = \varphi(h) \frac{\lambda+1}{2},
\end{displaymath}
so $\varphi$ must in fact be the evaluation homomorphism, proving injectivity. For the remaining part, let $\Delta := (\pi|^{\mathbb{D}})^{-1}(\mathbb{D})$. Then $\pi$ maps $\Delta$ homeomorphically onto $\mathbb{D}$, since the topology of $\Delta$ is the weak topology defined by Gel'fand transforms of functions from $\mathscr{M}(A^2_{(m)})$, and the topology of $\mathbb{D}$ is the weak topology defined by bounded functions in $\mathscr{M}(A^2_{(m)})$.
\end{proof}

The above theorem shows the existence of the analytic disk inside the character space $\mathfrak{M}(\mathscr{M}(A^2_{(m)}))$, and therefore it would be a natural question to ask whether this disk is dense therein, that is to see if the Corona Theorem could hold in this setting. The answer to this question is of course well-known and affirmative for $\mathscr{M}(H^2)=H^\infty$ and $\mathscr{M}(\mathcal{D})$, and thus one could make a similar conjecture about the multiplier space of $A^2_{(m)}$, but the techniques used to prove it for these two previous space are insufficient here, and it is not clear how their shortcomings could potentially be bypassed. Therefore, for now, it must remain an open question.

\textbf{Acknowledgment} 	The author of this article would like to thank the UK Engineering and Physical Research Council (EPSRC) and the School of Mathematics at the University of Leeds for their financial support. He is extremely grateful to Professor Jonathan R. Partington for all the help, guidance, and his useful comments.

\end{document}